\documentclass[reqno]{amsart}
\usepackage{amssymb, amsmath, amsthm, amscd, mathrsfs}

\usepackage{paralist}
\usepackage{cite}
\usepackage{bigints}
\usepackage{dsfont}
\usepackage{empheq}
\usepackage{amssymb}
\usepackage{cases}
\usepackage{enumitem}
\allowdisplaybreaks

\usepackage{caption}
\usepackage{booktabs}

\usepackage{float}
\usepackage[ruled,vlined,linesnumbered,noresetcount]{algorithm2e}

\usepackage{hyperref}
\numberwithin{equation}{section}

\theoremstyle{plain}
\newtheorem{theorem}{Theorem}

\newtheorem{lemma}{Lemma}
\newtheorem{proposition}{Proposition}
\theoremstyle{definition}
\newtheorem{remark}{Remark}

\newcommand{\vertiii}[1]{{\left\vert\kern-0.25ex\left\vert\kern-0.25ex\left\vert #1 
    \right\vert\kern-0.25ex\right\vert\kern-0.25ex\right\vert}}

\def \dv {\mathrm{div}}
\def \d {\mathrm{d}}

\title[Controllability and Inverse Problems for Parabolic Systems] %Use the shortened version of the full title
      {Controllability and Inverse Problems for Parabolic Systems with Dynamic Boundary Conditions}

\author{S. E. Chorfi}
\author{L. Maniar}

\address{S. E. Chorfi, L. Maniar, Faculty of Sciences Semlalia, Cadi Ayyad University, LMDP, IRD, UMMISCO, B.P. 2390, Marrakesh, Morocco}
\email{s.chorfi@uca.ac.ma, maniar@uca.ma}

\address{L. Maniar, The UM6P-Vanguard Center, University Mohammed VI Polytechnic, Benguerir, Morocco}
\email{Lahcen.Maniar@um6p.ma}

\makeatletter %added for 2020 MSC
\@namedef{subjclassname@2020}{%
  \textup{2020} Mathematics Subject Classification}
\makeatother

 
\begin{document}
\begin{abstract}
This review surveys previous and recent results on null controllability and inverse problems for parabolic systems with dynamic boundary conditions. We aim to demonstrate how classical methods such as Carleman estimates can be extended to prove null controllability for parabolic systems and Lipschitz stability estimates for inverse problems with dynamic boundary conditions of surface diffusion type. We mainly focus on the substantial difficulties compared to static boundary conditions. Finally, some conclusions and open problems will be mentioned.
\end{abstract}
\dedicatory{\large Dedicated to the memory of Professor Hammadi Bouslous}
\maketitle

\section{Introduction and motivation}
In this work, we review relevant results on parabolic systems with dynamic boundary conditions within the framework of null controllability and inverse problems, which are well-studied in the case of static boundary conditions (Dirichlet, Neumann, Robin). A typical example of these systems is given by
\begin{empheq}[left = \empheqlbrace]{alignat=2}
\begin{aligned}\label{eq:intro}
&\partial_{t}y -d\Delta y  = f  & \qquad\quad &\text{in } \Omega_T,\\
&\partial_{t}y_\Gamma  -\delta\Delta_\Gamma y_\Gamma +d\partial_\nu y  =  g  & &\text{on } \Gamma_T, \\
& y_\Gamma=y_{|\Gamma} & &\text{on } \Gamma_T, \\
&(y,y_\Gamma)|_{t=0}  = (y_0, y_{0,\Gamma})  & &\text{in } \Omega\times \Gamma,
\end{aligned}
\end{empheq}
where $T>0$, $\Omega\subset \mathbb{R}^N$ ($N\in\mathbb{N}$) is a bounded domain with smooth boundary $\Gamma = \partial\Omega$, $E_T:=(0,T)\times E$ for any subset $E\subset \mathbb{R}^N.$ Moreover, $y_0\in L^2(\Omega)$ and $y_{0,\Gamma}\in L^2(\Gamma)$ are, respectively, initial data in the bulk and on the boundary, which are not necessarily interrelated. The constants $d>0$ and $\delta\ge 0$ are the diffusion coefficients. In addition, $y_{|\Gamma}$ stands for the trace of a smooth function $y:\Omega\to \mathbb{R}$, $\nu$ is the outer unit normal field, $\partial_\nu$ is the normal derivative on $\Gamma$, and $\Delta_\Gamma$ denotes the Laplace-Beltrami operator on $\Gamma$. 

Dynamic boundary conditions refer to the presence of the time derivative of the state on the boundary (Equation $\eqref{eq:intro}_2$). They are also known as generalized Wentzell boundary conditions. In recent years, dynamic boundary conditions have garnered significant attention in both theoretical and applied communities; see for instance \cite{EMR,FGGR,Gal15,Hi89,SW,VV}. These boundary conditions are efficient in modeling several phenomena in applied sciences such as heat transfer, fluid mechanics, population dynamics, biology, and cell chemistry, among other fields; see the works \cite{Eg18,FH11,Glitzky,GM13,La32} for some intriguing applications. For a physical interpretation and derivation of these boundary conditions, readers can consult the seminal paper \cite{Gold}. Additionally, the paper \cite{Sa20} provides a new derivation utilizing the Carslaw-Jaeger relation.

\subsection{Literature on controllability problems with dynamic boundary conditions}
Up to our knowledge, the first work dealing with null controllability issues for parabolic equations with a dynamic boundary condition is \cite{CG12}. The authors developed a Carleman estimate involving an artificial boundary condition with the time derivative to prove the uniform boundary controllability of the corresponding system. However, they limited their study to the one-dimensional case and addressed some challenges of controllability in higher dimensions for such systems. The paper \cite{KN04} applied semigroup theory to achieve approximate controllability of a one-dimensional heat equation with dynamic boundary conditions via boundary control. An optimal control problem has been investigated in \cite{HKR}. Later, the authors of \cite{MMS17} demonstrated the distributed null controllability of multidimensional linear and semilinear heat equations with dynamic boundary conditions of surface diffusion type. They developed a new Carleman estimate specific to dynamic boundary conditions. Next, the paper \cite{KM20} extended the results to convection-diffusion equations with non-smooth drift terms both in the bulk and on its boundary. Furthermore, \cite{KMO22} investigated a similar problem for the semilinear heat equation with dynamic boundary conditions and general nonlinearities involving drift terms. The cost of approximate controllability has been studied in \cite{BCMO21}. Some time and norm optimal control problems have been considered in \cite{BMO22}. Later, hierarchical control problems within the framework of Stackelberg-Nash null controllability have been studied in \cite{BMO220}. Recently, \cite{CGKM23} examined the boundary null controllability of the heat equation with dynamic boundary conditions. The main result was achieved through a new boundary Carleman estimate and regularity estimates for the homogeneous adjoint system. The distributed null controllability of a stochastic counterpart has been proven in \cite{BBEM23}. Furthermore, discrete Carleman estimates for a fully discrete system applied to controllability have been developed in \cite{LMPZ23}. The paper \cite{CMM23} is devoted to the local null controllability of a cubic Ginzburg-Landau equation using a Carleman estimate for the linearized system along with an inverse function theorem. A similar approach has been used in \cite{Et24} for a general quasilinear system. In \cite{ZYG19}, the authors investigated an insensitizing control problem with $L^2$-norm of the state. The work \cite{SCM24} studies an insensitizing control problem involving tangential gradient terms. Regarding coupled systems, we refer to the recent work \cite{JM24}. As for impulsive controllability, we refer to the works \cite{CGMZ23I} and \cite{CGMZ22} by logarithmic convexity and Carleman commutator approaches. Partial results have been obtained for a singular equation with inverse square potential in \cite{BEOB23}. Finally, the paper \cite{CGMZ23F} investigated finite-time stabilization and impulsive controllability using logarithmic convexity and spectral techniques. Although we focus on parabolic equations, we refer to recent works \cite{MM23}, \cite{CGMZ23L}, and \cite{GT17} for the controllability of Schrödinger and wave equations with a dynamic boundary condition.

\subsection{Literature on inverse problems with dynamic boundary conditions}
There are few recent attempts to study inverse parabolic problems with dynamic boundary conditions. We mention the paper \cite{Sl15} for the identification of a time-dependent source from the knowledge of a space average using the backward Euler method, and \cite{Is18}, for the determination of a time-dependent source term from an integral overdetermination condition in a one-dimensional heat equation by using the generalized Fourier method. In \cite{ACMO21}, the authors have obtained a Lipschitz stability estimate for the inverse source problem via a new Carleman estimate for a general anisotropic system. We refer to the survey paper \cite{Ya09} for inverse problems with static boundary conditions by Carleman estimates. Next, the paper \cite{ACM22} has demonstrated Lipschitz stability for the radiative potentials and logarithmic stability for the initial temperatures using the logarithmic convexity method. An inverse coefficient problem has been considered in \cite{ACM22S}. The authors proved Lipschitz stability for interior and boundary potentials in a coupled semilinear system by a single observation component. As for numerical studies, we refer to \cite{ACM22I} for the identification of source terms by a gradient-type iterative method. In addition, the paper \cite{CGMZ23} has numerically studied the identification of unknown initial temperatures from the final overdetermination using a conjugate gradient method. The paper \cite{BFBI23} addressed the existence, uniqueness and a numerical method for an inverse problem from integral overdetermination involving a time-dependent heat source in a one-dimensional equation with nonlocal Wentzel-Neumann boundary conditions. Recently, the article \cite{Ra24} presented another numerical method for a similar problem using finite difference and spectral methods.

The remainder of the paper is organized as follows. In Section \ref{sec2}, we discuss the problems of well-posedness, distributed controllability, and boundary controllability associated with system \eqref{eq:intro}. Section \ref{sec3} is devoted to some inverse problems related to \eqref{eq:intro}. We first discuss the case of internal measurement and then give some partial new results for boundary measurement. Finally, we conclude our paper with Section \ref{sec4} which presents some open questions related to dynamic boundary conditions that deserve to be resolved.

\section{Null controllability} \label{sec2}
%\subsection{Well-posedness}
Let us start by briefly discussing the well-posedness of system \eqref{eq:intro} when $\delta>0$. The details can be found in \cite{MMS17}. The Lebesgue measure on $\Omega$ and the surface measure on $\Gamma$ will be respectively denoted by $\d x$ and $\d S$. We consider the real Hilbert spaces
\begin{align*}
\mathbb{L}^2 &:= L^2(\Omega)\times L^2(\Gamma), \qquad \mathbb{L}^2_T :=L^2(\Omega_T)\times L^2(\Gamma_T),\\
\mathbb{H}^k &:= \{(y,y_\Gamma)\in H^k(\Omega) \times H^k(\Gamma)\,:\, y_\Gamma=y_{|\Gamma}\}\quad\text{for }k\in\mathbb{N},\\
\mathbb{E}_1(t_0,t_1) &:= H^1\left(t_0,t_1; \mathbb{L}^2\right) \cap L^2\left(t_0,t_1; \mathbb{H}^2\right) \quad \text{for } t_1>t_0 \text{ in } \mathbb{R}, \qquad \mathbb{E}_1:=\mathbb{E}_1(0,T).
\end{align*}
The system \eqref{eq:intro} can be written as an abstract Cauchy problem
\begin{numcases}{\text{(ACP)\qquad}\label{acp}}
\hspace{-0.1cm} \partial_t Y=\mathcal{A} Y+ \mathcal{F}, \quad 0<t<T, \nonumber\\
\hspace{-0.1cm} Y(0)=\left(y_0, y_{0,\Gamma}\right), \nonumber
\end{numcases}
where $Y:=\left(y,y_{\Gamma}\right)$, $\mathcal{F}=(f,g)$ and the linear operator $\mathcal{A}$ is given by
$$D(\mathcal{A}) = \mathbb{H}^2, \qquad \mathcal{A} =\left ( \begin{array}{cc} d\Delta & 0 \\ -d\partial_\nu\; & \delta \Delta_\Gamma
\end{array}\right).
$$
The operator $\mathcal{A}$ is self-adjoint, non-positive and generates an analytic $C_0$-semigroup $\left(\mathrm{e}^{t\mathcal{A}}\right)_{t\geq 0}$ on $\mathbb{L}^2$. Therefore, the system \eqref{eq:intro} is well-posed and satisfies the maximal regularity in $\mathbb{L}^2$.

\subsection{Distributed control}
In this section, we discuss the distributed (or internal) null controllability of the controlled system with dynamic boundary conditions
\begin{empheq}[left = \empheqlbrace]{alignat=2}
\begin{aligned}\label{eq1}
&\partial_{t}y -d\Delta y  = \mathds{1}_\omega(x) v(t,x)   & \qquad\quad &\text{in } \Omega_T,\\
&\partial_{t}y_\Gamma  -\delta\Delta_\Gamma y_\Gamma +d\partial_\nu y  =  0  & &\text{on } \Gamma_T, \\
& y_\Gamma=y_{|\Gamma} & &\text{on } \Gamma_T, \\
&(y,y_\Gamma)|_{t=0} = (y_0, y_{0,\Gamma})  & &\text{in } \Omega\times \Gamma,
\end{aligned}
\end{empheq}
where the control region $\omega \Subset \Omega$ is an 
arbitrary nonempty open subset which is strictly contained in $\Omega$ (i.e., $\overline{\omega}
\subset \Omega$), and $\mathds{1}_\omega$ the indicator function of $\omega$. We seek a control $v\in L^2(\omega_T)$ such that the solution of \eqref{eq1} satisfies
\begin{equation}\label{eqc}
    y(T,\cdot) =0\; \text{ in } \Omega \quad \text{ and } \quad y_\Gamma(T,\cdot) =0\; \text{ on } \Gamma.
\end{equation}
We state the main result on the null controllability of \eqref{eq1}.
\begin{theorem} \label{thm:intro} 
Assume that $\delta>0$. For all $T > 0$, all nonempty open set $\omega \Subset \Omega$ and all initial data
$(y_0, y_{0,\Gamma})\in \mathbb{L}^2$ there exists a control $v\in L^2(\omega_T)$ such that the unique mild solution $(y,y_{\Gamma})$ of \eqref{eq1} satisfies \eqref{eqc}.
\end{theorem}
By a classical duality argument (see \cite{TW}), the above null controllability result is equivalent to an observability inequality for the homogeneous backward system
\begin{empheq}[left = \empheqlbrace]{alignat=2}
\begin{aligned}\label{eq2}
&-\partial_t \varphi-d\Delta \varphi  = 0&  \qquad \quad &\text{in }\Omega_T,\\
&-\partial_t \varphi_\Gamma -\delta\Delta_\Gamma \varphi_\Gamma +d\partial_\nu \varphi =0 &  &\text{on }\Gamma_T,\\
& \varphi_\Gamma=\varphi_{|\Gamma} & &\text{on } \Gamma_T, \\
&(\varphi,\varphi_\Gamma)|_{t=T} = (\varphi_T, \varphi_{T,\Gamma}) & & \text{in }\Omega \times \Gamma.
\end{aligned}
\end{empheq}
That is, the following result.
\begin{proposition}\label{lem:observa} 
Assume that $\delta>0$. There exists a constant $C >0$ such that for all $(\varphi_T, \varphi_{T,\Gamma})\in \mathbb{L}^2$ the mild solution $(\varphi, \varphi_\Gamma)$ of the backward system \eqref{eq2} satisfies the following observability inequality
\begin{equation}
 \|(\varphi(0,\cdot),\varphi_\Gamma(0,\cdot))\|_{\mathbb L^2}^2  \leq C \int_{\omega_T} |\varphi|^2 \,\d x\,\d t.\label{eq:observa}
\end{equation}
 \end{proposition}
The key tool to show the above observability inequality is the following global Carleman estimate (see \cite{FI96,FG06} for static boundary conditions). Let us first recall some weight functions. Let $\omega' \Subset \Omega$ be a nonempty open subset. There exists a function $\eta\in C^2\left(\overline{\Omega}\right)$ such that (see \cite[Lemma 1.1]{FI96}):
\begin{align*}
&{\rm i)}\quad \eta> 0  \; \text{ in } \Omega, \quad \eta=0 \; \text{ on } \Gamma,\\
&{\rm ii)}\quad |\nabla\eta|  > 0 \quad\text{ in } \overline{\Omega\backslash\omega'},\\
&{\rm iii)}\quad \partial_\nu \eta \leq -c <0 \quad\text{ on } \Gamma
\end{align*}
for some constant $c>0$. Moreover, we consider \begin{equation}\label{w1}
\alpha(t,x)=\frac{\mathrm{e}^{2\lambda \|\eta\|_\infty}- \mathrm{e}^{\lambda \eta(x)}}{t(T -t)} \quad \text{ and } \quad \xi(t,x)=\frac{\mathrm{e}^{\lambda \eta(x)}}{t(T -t)}
\end{equation}
for all $(t,x)\in \overline{\Omega}_T$ and $\lambda \geq 1$ is a parameter (to be fixed later) which depends only on $\Omega$ and $\omega$.
Note that the functions $\alpha$ and $\xi$ are of class $C^2$, strictly positive on $\overline{\Omega}_T$ and blow up as $t\to 0$ and $t\to T$.
\begin{lemma}[Carleman estimate]\label{lem:carleman} 
Assume that $\delta>0$. Let $T>0$, $\omega \Subset  \Omega$ be nonempty and open. Then there exist constants $C>0$ and $\lambda_1,s_1 \ge 1$ such that 
{\small\begin{align}
 s^{-1} \int_{\Omega_T} &e^{-2s\alpha} \xi^{-1} (|\partial_t\varphi|^2 + |\Delta \varphi|^2)\,\d x\,\d t 
  + s^{-1}\int_{\Gamma_T}e^{-2s\alpha}\xi^{-1}(|\partial_t \varphi_\Gamma|^2 
   + |\Delta_\Gamma \varphi_\Gamma|^2)\,\d S\,\d t \notag\\
  & \qquad + s\lambda^2 \int_{\Omega_T} e^{-2s\alpha} \xi |\nabla \varphi|^2 \,\d x\,\d t + s\lambda \int_{\Gamma_T} e^{-2s\alpha} \xi |\nabla_\Gamma \varphi_\Gamma|^2 \,\d S\,\d t \notag\\
 &  \qquad + s^3\lambda^4\int_{\Omega_T} e^{-2s\alpha} \xi^3 |\varphi|^2 \, \d x\,\d t +  s^3\lambda^3 \int_{\Gamma_T} e^{-2s\alpha} \xi^3 |\varphi_\Gamma|^2 \, \d S\,\d t \notag\\
  & \qquad + s\lambda \int_{\Gamma_T} e^{-2s\alpha} \xi|\partial_\nu \varphi|^2 \,\d S\,\d t  \notag\\
&\leq C s^3\lambda^4 \int_{\omega_T} e^{-2s\alpha} \xi^3 |\varphi|^2 \, \d x\,\d t 
 + C \int_{\Omega_T} e^{-2s\alpha} |\partial_t\varphi \pm d\Delta \varphi|^2 \,\d x \,\d t  \notag\\
  &\qquad + C \int_{\Gamma_T} e^{-2s\alpha} |\partial_t\varphi_\Gamma \pm \delta \Delta_\Gamma \varphi_\Gamma 
    - d \partial_\nu \varphi|^2 \,\d S\,\d t \label{car1}
\end{align}}
for all $\lambda \geq \lambda_1$, $s \geq s_1$ and all $(\varphi,\varphi_\Gamma)\in \mathbb{E}_1$.
\end{lemma}
The proof follows the classical strategy of Carleman estimates for parabolic equations, but one needs to handle several new boundary terms carefully; see \cite{MMS17} for the details. Next, we focus on the main difficulties.
\begin{remark}
Absorbing the boundary term
\begin{equation}\label{termc}
    s\lambda\int_{\Gamma_T} e^{-2s\alpha}\xi |\partial_\nu \eta| \, |\nabla_\Gamma \varphi_\Gamma|^2\,\d S\,\d t
\end{equation}
into the left-hand side of the Carleman estimate is quite challenging, since it appears on both sides of the estimate with the same exponents of the parameters $s$ and $\lambda$. The assumption $\delta > 0$ is needed to absorb \eqref{termc} using the equation $\Delta_\Gamma \varphi_\Gamma=\frac{1}{\delta}(-\partial_t \varphi_\Gamma +d\partial_\nu \varphi)$. Note that the term \eqref{termc} does not appear in the 1D case.
\end{remark}

\begin{remark}
The Carleman estimate of Lemma \ref{lem:carleman} can be extended for general time-space dependent potentials, diffusion matrices, and drift terms of gradient type; see \cite{ACMO21} and \cite{KMO22}.
\end{remark}

\subsection{Boundary control}
Now we discuss the boundary null controllability of the controlled system
\begin{empheq}[left = \empheqlbrace]{alignat=2}
\begin{aligned}\label{eq3}
&\partial_{t}y -d\Delta y = 0  & \qquad\quad &\text{in } \Omega_T,\\
&\partial_{t}y_\Gamma  -\delta\Delta_\Gamma y_\Gamma +d\partial_\nu y  =  \mathds{1}_\gamma(x) v(t,x)  & &\text{on } \Gamma_T, \\
& y_\Gamma=y_{|\Gamma} & &\text{on } \Gamma_T, \\
&(y,y_\Gamma)|_{t=0} = (y_0, y_{0,\Gamma})  & &\text{in } \Omega\times \Gamma,
\end{aligned}
\end{empheq}
where the control is supported in a sub-boundary $\gamma \subset\Gamma$ which is an arbitrary relatively open subset.

Similarly to the distributed controllability, we have the following boundary null controllability concerning the system \eqref{eq3}.
\begin{theorem} \label{thmobs} 
Assume that $\delta>0$. For all $T > 0$, all nonempty relatively open set $\gamma \Subset \Gamma$ and  all initial data
$(y_0, y_{0,\Gamma})\in \mathbb{L}^2$ there exists a control $v\in L^2(\gamma_T)$ 
such that the unique mild solution $(y,y_\Gamma)$ of \eqref{eq3} satisfies
$$y(T,\cdot) =0\; \text{ in } \Omega \quad \text{ and } \quad y_\Gamma(T,\cdot) =0\; \text{ on } \Gamma.$$
\end{theorem}
This boundary null controllability result is equivalent to the following boundary observability inequality.
\begin{proposition}\label{lem:observa1} 
Assume that $\delta>0$. There exists a constant $C >0$ such that for all $(\varphi_T, \varphi_{T,\Gamma})\in \mathbb{L}^2$ the mild solution $(\varphi,\varphi_\Gamma)$ of the backward system \eqref{eq2} satisfies the following observability inequality
\begin{equation}
 \|(\varphi(0,\cdot),\varphi_\Gamma(0,\cdot))\|_{\mathbb L^2}^2  \leq C \int_{\gamma_T} |\varphi_\Gamma|^2 \,\d S\,\d t.\label{eq:observa2}
\end{equation}
 \end{proposition}
As before, the key ingredient to show the above observability inequality is a boundary Carleman estimate. Here, we need a weight function with modified properties, namely, there exists a function 
$\eta_0\in C^2(\overline{\Omega})$ such that (see \cite{Im95})
\begin{align*}
&{\rm i)} \quad\eta_0 > 0 \text{ in }\Omega, \quad \eta_0=0 \text{ on }\Gamma\setminus \gamma,\\
&{\rm ii)} \quad |\nabla \eta_0| \geq c_0> 0 \text{ in }\overline{\Omega},\\
&{\rm iii)} \quad \partial_\nu \eta_0 \leq -c_0 \text{ on }\Gamma \setminus \gamma
\end{align*}
for some constant $c_0 > 0$. Note that these properties differ from those of the internal weight function $\eta$. Moreover, we consider the weights $\alpha$ and $\xi$ as in \eqref{w1} replacing $\eta$ by $\eta_0$.
\begin{lemma}[Boundary Carleman estimate]\label{car2} 
Assume that $\delta>0$. Let $T>0$, $\gamma \Subset  \Gamma$ be nonempty and relatively open. Then there exist constants $C>0$ and $\lambda_1,s_1 \ge 1$ such that any smooth solution of the adjoint system \eqref{eq2} satisfies
{\small\begin{align*}
 s^{-1} \int_{\Omega_T} &e^{-2s\alpha} \xi^{-1} (|\partial_t\varphi|^2 + |\Delta \varphi|^2)\,\d x\,\d t 
  + s^{-1}\int_{\Gamma_T}e^{-2s\alpha}\xi^{-1}(|\partial_t \varphi_\Gamma|^2 
   + |\Delta_\Gamma \varphi_\Gamma|^2)\,\d S\,\d t \\
  & \qquad + s\lambda^2 \int_{\Omega_T} e^{-2s\alpha} \xi |\nabla \varphi|^2 \,\d x\,\d t + s\lambda \int_{\Gamma_T} e^{-2s\alpha} \xi |\nabla_\Gamma \varphi_\Gamma|^2 \,\d S\,\d t \\
 &  \qquad + s^3\lambda^4\int_{\Omega_T} e^{-2s\alpha} \xi^3 |\varphi|^2 \, \d x\,\d t +  s^3\lambda^3 \int_{\Gamma_T} e^{-2s\alpha} \xi^3 |\varphi_\Gamma|^2 \, \d S\,\d t \\
  & \qquad + s\lambda \int_{\Gamma_T} e^{-2s\alpha} \xi|\partial_\nu \varphi|^2 \,\d S\,\d t  \\
&\leq C s^6\lambda^6 \int_{\gamma_T} e^{-2s\alpha} \xi^6 |\varphi_\Gamma|^2 \, \d x\,\d t
\end{align*}}
for all $\lambda \geq \lambda_1$ and $s \geq s_1$.
\end{lemma}
For the proof, we refer to \cite{CGKM23} for a detailed exposition.
\begin{remark}\label{rmkbc}
The boundary term
$$
s\lambda \int_{\gamma_T} \mathrm{e}^{-2s\alpha} \xi (\partial_\nu \eta_0) |\partial_\nu \varphi|^2 \, \d S\, \d t
$$
is more challenging in the boundary control case because $\partial_\nu \eta_0 <0$ only on $\Gamma \setminus \gamma$, while in the internal Carleman estimate \eqref{car1}, $\partial_\nu \eta <0$ on the whole $\Gamma$. To overcome this issue, one trick is to replace the term $\partial_\nu \varphi$ by the equation $\eqref{eq2}_2$ and use higher regularity estimates. However, this technique prevents obtaining a boundary Carleman estimate with general $L^2$ source terms as in the distributed control case.
\end{remark}
Once the lateral data are available, that is, Dirichlet and Neumann data, one can prove the following Carleman estimate with source terms.
\begin{lemma}\label{car3} 
Assume that $\delta>0$. There exist constants $C>0$ and $\lambda_1,s_1 \ge 1$ such that
{\small\begin{align*}
 s^{-1} \int_{\Omega_T} &e^{-2s\alpha} \xi^{-1} (|\partial_t z|^2 + |\Delta z|^2)\,\d x\,\d t 
  + s^{-1}\int_{\Gamma_T}e^{-2s\alpha}\xi^{-1}(|\partial_t z_\Gamma|^2 
   + |\Delta_\Gamma z_\Gamma|^2)\,\d S\,\d t \\
  & \qquad + s\lambda^2 \int_{\Omega_T} e^{-2s\alpha} \xi |\nabla z|^2 \,\d x\,\d t + s\lambda \int_{\Gamma_T} e^{-2s\alpha} \xi |\nabla_\Gamma z_\Gamma|^2 \,\d S\,\d t \\
 &  \qquad + s^3\lambda^4\int_{\Omega_T} e^{-2s\alpha} \xi^3 |z|^2 \, \d x\,\d t +  s^3\lambda^3 \int_{\Gamma_T} e^{-2s\alpha} \xi^3 |z_\Gamma|^2 \, \d S\,\d t \\
  & \qquad + s\lambda \int_{\Gamma_T} e^{-2s\alpha} \xi|\partial_\nu z|^2 \,\d S\,\d t  \\
&\hspace{-0.3cm}\leq C s^3\lambda^4 \int_{\gamma_T} e^{-2s\alpha} \xi^3 |z_\Gamma|^2 \, \d S\,\d t + s\lambda \int_{\gamma_T} e^{-2s\alpha} \xi|\partial_\nu z|^2 \,\d S\,\d t \notag\\
&\hspace{-0.3cm} + C \int_{\Omega_T} e^{-2s\alpha} |\partial_t z - d\Delta z|^2 \,\d x \,\d t + C \int_{\Gamma_T} e^{-2s\alpha} |\partial_t z_\Gamma - \delta \Delta_\Gamma z_\Gamma 
    + d \partial_\nu z|^2 \,\d S\,\d t
\end{align*}}
for all $\lambda \geq \lambda_1$, $s \geq s_1$ and all $(z,z_\Gamma)\in \mathbb{E}_1$.
\end{lemma}

\section{Inverse problems}\label{sec3}
Henceforth $\delta>0$. Consider the following system with dynamic boundary conditions
\begin{empheq}[left = \empheqlbrace]{alignat=2}
\begin{aligned}\label{eq5}
&\partial_{t}y -d\Delta y = f(t,x)   & \qquad\quad &\text{in } \Omega_T,\\
&\partial_{t}y_\Gamma  -\delta\Delta_\Gamma y_\Gamma +d\partial_\nu y  =  g(t,x)  & &\text{on } \Gamma_T, \\
& y_\Gamma=y_{|\Gamma} & &\text{on } \Gamma_T, \\
&(y,y_\Gamma)|_{t=0} = (y_0, y_{0,\Gamma})  & &\text{in } \Omega\times \Gamma.
\end{aligned}
\end{empheq}
We shall focus on inverse source problems, whereas other inverse problems could be considered such as inverse potential (zeroth order coefficients) and initial data problems.

Let $t_0 \in (0,T)$, $\displaystyle T_0=\frac{T+t_0}{2}$,  $E_{t_0,T}=\left(t_0,T\right)\times E$, and $Y:=(y,y_\Gamma)$.
\subsection{Interior measurement}
We seek the determination of unknown source terms $(f,g)$ in \eqref{eq5} from the interior measurement
$$y(t,x)|_{\omega_{t_0,T}}\; (\text{first component}) \quad \text{ and } \quad Y(T_0,\cdot).$$
The main questions to be resolved in this context are the following:
\begin{itemize}
    \item[(i)] Uniqueness: does $y|_{\omega_{t_0,T}}=0 \text{ and } Y(T_0,\cdot)=0$ imply $(f,g)=(0,0)$?
    \item[(ii)] Stability: can one obtain continuous dependence, namely, continuity of the correspondence $\left(y|_{\omega_{t_0,T}}, Y(T_0,\cdot)\right) \mapsto (f,g)$ for suitable norms?
    \item[(iii)] Identification: can one design a numerical algorithm to reconstruct the unknowns $(f,g)$ from given measurements?
\end{itemize}
For general source terms $(f,g)\in \mathbb{L}^2_T$, nonuniqueness occurs and can be seen as a consequence of null controllability presented in Section \ref{sec2}. More precisely, we show the following result.
\begin{proposition}
For any $Y_0 :=\left(y_0,y_{0,\Gamma}\right)\in \mathbb{L}^2$, there exists a non-zero source term $f\in L^2\left(\Omega_T\right)$ such that the mild solution $Y$ of \eqref{eq5} with $g=0$ satisfies $y=0$ in $\omega_{t_0,T}$ and $Y\left(T_0,\cdot\right)=0$.
\end{proposition}

\begin{proof}
Let $t' \in (0,t_0)$ and consider a non-zero source function $f_1\in L^2\left(\Omega_{t'}\right)$. Then the corresponding solution $Y_{f_1} =\left(y_{f_1} ,y_{f_1,\Gamma}\right)$ to \eqref{eq5} in the time interval $(0,t')$ is not identically null. Choosing as initial condition $Y_1=\left(y_{f_1} (t', \cdot), 0\right)$, there exists a control $v\in L^2\left((t', t_0) \times \omega)\right)$ such that $Y_v\left(t_0 , \cdot\right) \equiv 0$ in $\Omega\times \Gamma$, as a consequence of Theorem \ref{thm:intro}. Consider $f=\mathds{1}_{(0, t')} f_1 + \mathds{1}_{(t', t_0)} v \in L^2\left(\Omega_T\right)$ and let $Y$ denote the corresponding solution of \eqref{eq5}. We have $f\not\equiv 0$, while $y=0$ in $\omega_{t_0,T}$ and $Y\left(T_0,\cdot\right)=0$.
\end{proof}
Therefore, stability fails in general, and one needs to restrict the source terms to some admissible set to obtain stability results. Let us introduce the following admissible set
\begin{align}\label{eqas}
\mathcal{S}\left(C_0\right) :=\left\{(f,g)\in H^1\left(0,T; \mathbb{L}^2\right): \begin{array}{ll}
\left|f_t(t,x)\right| \leq C_0 \left|f(T_0,x)\right|, &\text{ a.e. }(t,x)\in \Omega_T \\
\left|g_t(t,x)\right| \leq C_0 \left|g(T_0,x)\right|, &\text{ a.e. }(t,x)\in \Gamma_T
\end{array}\right\}
\end{align}
for a given $C_0>0$. The set $\mathcal{S}\left(C_0\right)$ includes an interesting class of source terms, e.g., some source terms of the form $f(t,x)=f_0(x)r(t,x)$ and $g(t,x)=g_0(x)s(t,x)$. 

The main result on the Lipschitz stability estimate for the inverse source problem reads as follows.
\begin{theorem} \label{thmstab1} 
There exists a positive constant $C=C\left(\Omega, \omega,T,t_0,C_0\right)$ such that, for any admissible source $\mathcal{F}=(f,g) \in \mathcal{S}\left(C_0\right)$, we have
\begin{equation}\label{eq3.71}
\|(f,g)\|_{\mathbb{L}^{2}_T} \leq C\left(\|\partial_t y\|_{L^2\left(\omega_{t_0, T}\right)}+\left\|Y(T_0,\cdot)\right\|_{\mathbb{H}^2}\right),
\end{equation}
where $Y :=\left(y,y_\Gamma\right)$ is the mild solution of \eqref{eq5}.
\end{theorem}
We refer to the next section for the key ideas of the proof.

\begin{remark}
It should be emphasized that the above stability estimate only uses one measurement component of the solution $Y$, localized in the interior of $\Omega$, namely $y|_{\omega_{t_0,T}}$. This is possible thanks to the internal Carleman estimate with source terms and some trace theorems.
\end{remark}

\subsection{Boundary measurement}
Here, we seek the determination of unknown source terms $(f,g)$ in \eqref{eq5} from boundary measurements
$$y(t,x)|_{\gamma_{t_0,T}}, \quad \partial_\nu y(t,x)|_{\gamma_{t_0,T}} \quad \text{ and } \quad Y(T_0,\cdot).$$
We will prove a new Lipschitz stability estimate using the Carleman estimate of Lemma \ref{car3} and following the Bukhgeim-Klibanov method \cite{BK81}; see also \cite{IY98}.
\begin{theorem} \label{thmstab} 
There exists a positive constant $C=C\left(\Omega, \gamma,T,t_0,C_0\right)$ such that, for any admissible source $\mathcal{F}=(f,g) \in \mathcal{S}\left(C_0\right)$, we have
\begin{equation}\label{eq3.7}
\|(f,g)\|_{\mathbb{L}^{2}_T} \leq C\left(\|\partial_t y_\Gamma\|_{L^2(\gamma_{t_0, T})}+ \|\partial_t\partial_\nu y\|_{L^2\left(\gamma_{t_0, T}\right)}+\left\|Y(T_0,\cdot)\right\|_{\mathbb{H}^2}\right),
\end{equation}
where $Y :=\left(y,y_\Gamma\right)$ is the mild solution of \eqref{eq5}.
\end{theorem}

\begin{remark}
In contrast to the interior measurement case in Theorem \ref{thmstab1}, we have used two lateral Cauchy data for the boundary measurement case. This is due to the restriction on the boundary Carleman estimate as seen in Remark \ref{rmkbc}. See Section \ref{sec4} for further comments.
\end{remark}

\begin{proof}[Proof of Theorem \ref{thmstab}]
Rescaling time, it is sufficient to prove the estimate \eqref{eq3.7} in the case $T_0=\frac{T}{2}$. Throughout the proof, $C$ will denote a generic constant that is independent of $Y$. The terms appearing in \eqref{eq3.7} are well defined, since $Y:=(y,y_{\Gamma})\in \mathbb{E}_1$. The function $(z,z_\Gamma)=(\partial_t y, \partial_t y_\Gamma)$, where $(y,y_\Gamma)$ is the solution of \eqref{eq5}, satisfies the following system
\begin{empheq}[left = \empheqlbrace]{alignat=2}
\begin{aligned}
&\partial_t z - d\Delta z = f_t(t,x), &\quad\text{in } \Omega_T , \\
&\partial_t z_{\Gamma} - \delta \Delta_\Gamma z_\Gamma +d\partial_{\nu} z = g_t(t,x), &\quad\text{on } \Gamma_T, \\
&z_{\Gamma} = z_{|\Gamma}, &\quad\text{on } \Gamma_T, \label{eq3.8to3.10}
\end{aligned}
\end{empheq}
and we have
\begin{empheq}[left = \empheqlbrace]{alignat=2}
& z\left(T_0\right) - d\Delta y(T_0) = f\left(T_0\right), \label{eq3.11}\\
& z_{\Gamma}\left(T_0\right) -\delta \Delta_\Gamma y_{\Gamma}\left(T_0\right) + d\partial_{\nu} y\left(T_0\right)= g\left(T_0\right).\label{eq3.12}
\end{empheq}
Since $(f,g)\in H^1\left(0,T; \mathbb{L}^2\right)$, by the maximal regularity we have $\left(z, z_\Gamma\right) \in \mathbb{E}_1\left(t_0, T\right)$. Hence, we may apply the Carleman estimate of Lemma \ref{car3} to \eqref{eq3.8to3.10}; we obtain
\begin{small}
\begin{align}
& \int_{\Omega_{t_0,T}} \left(\frac{1}{s\xi} |\partial_t z|^2 + s^3\lambda^4 \xi^3 |z|^2 \right)\mathrm{e}^{-2s\alpha} \d x \d t + \int_{\Gamma_{t_0,T}} \left(\frac{1}{s\xi} |\partial_t z_\Gamma|^2 + s^3\lambda^3\xi^3 |z_\Gamma|^2 \right)\mathrm{e}^{-2s\alpha} \d S \d t \nonumber\\
& \quad\leq C s^3\lambda^4\int_{\gamma_{t_0,T}} \mathrm{e}^{-2s\alpha} \xi^3|z_\Gamma|^2 \d S \d t + C s\lambda\int_{\gamma_{t_0,T}} \mathrm{e}^{-2s\alpha} \xi |\partial_\nu z|^2 \d S \d t\nonumber\\
&\quad +C \int_{\Omega_{t_0,T}} \mathrm{e}^{-2s\alpha}|f_t|^2 \d x \d t + C \int_{\Gamma_{t_0,T}} \mathrm{e}^{-2s\alpha} |g_t|^2 \d S \d t \label{eqq3.12}
\end{align}
\end{small}
for any $s>0$ large enough.
Since $\mathcal{F}=(f,g)\in \mathcal{S}\left(C_0\right)$, we have
\begin{small}
\begin{align}
& \int_{\Omega_{t_0,T}} \left(\frac{1}{s\xi} |\partial_t z|^2 + s^3\lambda^4\xi^3 |z|^2 \right)\mathrm{e}^{-2s\alpha} \d x \d t + \int_{\Gamma_{t_0,T}} \left(\frac{1}{s\xi} |\partial_t z_\Gamma|^2 + s^3\lambda^3\xi^3 |z_\Gamma|^2 \right)\mathrm{e}^{-2s\alpha} \d S \d t \nonumber\\
& \quad\leq C s^3\lambda^4\int_{\gamma_{t_0,T}} \mathrm{e}^{-2s\alpha} \xi^3|z_\Gamma|^2 \d S \d t + C s\lambda\int_{\gamma_{t_0,T}} \mathrm{e}^{-2s\alpha} \xi|\partial_\nu z|^2 \d S \d t\nonumber\\
& \quad +C \int_{\Omega_{t_0,T}} \mathrm{e}^{-2s\alpha}\left|f\left(T_0,x\right)\right|^2 \,\d x\,\d t + C \int_{\Gamma_{t_0,T}} \mathrm{e}^{-2s\alpha} \left|g\left(T_0,x\right)\right|^2 \,\d S\,\d t. \label{eq3.13}
\end{align}
\end{small}
From \eqref{eq3.11}-\eqref{eq3.12}, we can estimate the term
\begin{equation*}
\int_\Omega \left|f\left(T_0,x\right)\right|^2 \mathrm{e}^{-2s\alpha\left(T_0,x\right)} \,\d x + \int_\Gamma \left|g\left(T_0,x\right)\right|^2 \mathrm{e}^{-2s\alpha\left(T_0,x\right)} \,\d S,
\end{equation*}
by estimating the term
\begin{equation*}
\int_\Omega \left|z\left(T_0,x\right)\right|^2 \mathrm{e}^{-2s\alpha\left(T_0,x\right)} \,\d x + \int_\Gamma \left|z_\Gamma\left(T_0,x\right)\right|^2 \mathrm{e}^{-2s\alpha\left(T_0,x\right)} \,\d S.
\end{equation*}
We have
\begin{align*}
&\int_\Omega \left|z\left(T_0,x\right)\right|^2 \mathrm{e}^{-2s\alpha\left(T_0,x\right)} \,\d x = \int_{t_0}^{T_0} \frac{\partial}{\partial t} \left(\int_\Omega \left|z\left(t,x\right)\right|^2 \mathrm{e}^{-2s\alpha(t,x)} \,\d x\right)\,\d t\\
&= \int_{t_0}^{T_0} \int_\Omega \left(2\partial_t y(t,x)\partial_t^2 y(t,x)- 2s(\partial_t \alpha)|\partial_t y(t,x)|^2 \right) \mathrm{e}^{-2s\alpha} \,\d x\,\d t\\
&\leq \int_{\Omega_{t_0,T}} \left(2|\partial_t y(t,x)|\,|\partial_t^2 y(t,x)| +Cs\xi^2 |\partial_t y(t,x)|^2 \right) \mathrm{e}^{-2s\alpha} \,\d x\,\d t,
\end{align*}
where we employed $|\partial_t \alpha| \leq C\xi^2$. By Young's inequality, we have
\begin{align*}
2|\partial_t y(t,x)|\,|\partial_t^2 y(t,x)| &\leq C\left(\frac{1}{s^2\xi} |\partial_t^2 y(t,x)|^2 + s^2 \lambda^4\xi^3 |\partial_t y(t,x)|^2\right)
\end{align*}
for large $\lambda$, using $\xi\leq C\xi^2$.
Hence,
\begin{equation}\label{E1}
\int_\Omega \left|z\left(T_0,x\right)\right|^2 \mathrm{e}^{-2s\alpha\left(T_0,x\right)} \,\d x \leq C\int_{\Omega_{t_0,T}} \left(\frac{1}{s^2\xi} |\partial_t^2 y|^2 + s^2\lambda^4\xi^3 |\partial_t y|^2 \right)\mathrm{e}^{-2s\alpha} \,\d x\,\d t.
\end{equation}
Similarly, we obtain
\begin{equation}\label{E2}
\int_\Gamma \left|z_\Gamma\left(T_0,x\right)\right|^2 \mathrm{e}^{-2s\alpha\left(T_0,x\right)} \,\d S \leq C\int_{\Gamma_{t_0,T}} \left(\frac{1}{s^2\xi} |\partial_t^2 y_\Gamma|^2 + s^2\lambda^3\xi^3 |\partial_t y_\Gamma|^2 \right)\mathrm{e}^{-2s\alpha} \,\d S\,\d t
\end{equation}
for large $\lambda$.
Using \eqref{eq3.13}-\eqref{E2}, we obtain
\begin{align}
& \int_\Omega \left|z\left(T_0,x\right)\right|^2 \mathrm{e}^{-2s\alpha\left(T_0,x\right)} \,\d x + \int_\Gamma \left|z_\Gamma\left(T_0,x\right)\right|^2 \mathrm{e}^{-2s\alpha\left(T_0,x\right)} \,\d S \nonumber\\
& \leq C\int_{\Omega_{t_0,T}} \left(\frac{1}{s^2\xi} |\partial_t^2 y|^2 + s^2\lambda^4\xi^3 |\partial_t y|^2 \right)\mathrm{e}^{-2s\alpha} \,\d x\,\d t \nonumber\\
& \quad + C\int_{\Gamma_{t_0,T}} \left(\frac{1}{s^2\xi} |\partial_t^2 y_\Gamma|^2 + s^2\lambda^3\xi^3 |\partial_t y_\Gamma|^2 \right)\mathrm{e}^{-2s\alpha} \,\d S\,\d t \nonumber\\
& \leq \frac{C}{s} \int_{\Omega_{t_0,T}} \mathrm{e}^{-2s\alpha} \left|f\left(T_0,x\right)\right|^2 \,\d x\,\d t + \frac{C}{s} \int_{\Gamma_{t_0,T}} \mathrm{e}^{-2s\alpha} \left|g\left(T_0,x\right)\right|^2 \,\d S\,\d t \nonumber\\
& \quad + \,Cs^2\lambda^4 \int_{\gamma_{t_0,T}} \mathrm{e}^{-2s\alpha}\xi^3 |\partial_t y|^2 \,\d S\,\d t +  C\lambda \int_{\gamma_{t_0,T}} \mathrm{e}^{-2s\alpha}\xi |\partial_t\partial_\nu y|^2 \,\d S\,\d t\label{EE2}
\end{align}
It can be seen that
\begin{align}
&\int_\Omega \left|\Delta y\left(T_0,\cdot\right)\right|^2 \mathrm{e}^{-2s\alpha\left(T_0,\cdot\right)} \,\d x \leq C \left\Vert y\left(T_0,\cdot\right)\right\Vert_{H^2\left(\Omega\right)}^2, \label{EE1}\\
&\int_\Gamma \left|\Delta_\Gamma  y_\Gamma\left(T_0,\cdot\right)\right|^2\mathrm{e}^{-2s\alpha\left(T_0,\cdot\right)} \,\d S + \int_\Gamma \left|\partial_\nu y\left(T_0,\cdot\right)\right|^2 \mathrm{e}^{-2s\alpha\left(T_0,\cdot\right)} \,\d S \nonumber\\
&\leq C\left\Vert Y\left(T_0,\cdot\right)\right\Vert_{\mathbb{H}^2}^2 \label{EEE2},
\end{align}
using the trace theorem. By \eqref{EE2} and $\alpha(t,x)\geq \alpha(T_0,x)$ for all $(t,x)\in \overline{\Omega}_{T}$, we can deduce
\begin{align}
&\left(1-\frac{C}{s}\right)\left(\int_{\Omega} \mathrm{e}^{-2s\alpha\left(T_0,x\right)} \left|f\left(T_0,x\right)\right|^2 \,\d x + \int_{\Gamma} \mathrm{e}^{-2s\alpha\left(T_0,x\right)} \left|g\left(T_0,x\right)\right|^2 \,\d S \right)\nonumber\\
& \leq C\left\Vert Y\left(T_0,\cdot\right)\right\Vert_{\mathbb{H}^2}^2+ C\lambda \int_{\gamma_{t_0,T}} \mathrm{e}^{-2s\alpha}\xi \left(s^2\lambda^3\xi^2|\partial_t y_\Gamma|^2 + |\partial_t\partial_\nu y|^2\right) \,\d S\,\d t \label{EE5}
\end{align}
Since $(f,g)\in \mathcal{S}(C_0)$, we infer that
\begin{align*}
\begin{aligned}
|f(t,x)|\leq \left|f\left(T_0,x\right)\right|+ \left|\int_{T_0}^t f_\tau(\tau,x) \,\d \tau \right| \leq C\left|f\left(T_0,x\right)\right| \; \text{ for a.e } (t,x)\in \Omega_T ,\\
|g(t,x)|\leq \left|g\left(T_0,x\right)\right|+ \left|\int_{T_0}^{t} g_\tau(\tau,x) \,\d \tau \right| \leq C\left|g\left(T_0,x\right)\right| \; \text{ for a.e } (t,x)\in \Gamma_T .
\end{aligned}
\end{align*}
Plugging the above inequalities in \eqref{EE5} and fixing $\lambda,s>0$ sufficiently large, the proof of Theorem \ref{thmstab} is achieved.
\end{proof}

\section{Conclusion and open problems}\label{sec4}
This section is devoted to certain concluding comments and open questions that deserve future investigation.

\subsection*{Open problem 4.1 (The case $\delta=0$)}
The null controllability of the higher-dimensional system \eqref{eq1} with dynamic boundary conditions and without surface diffusion, i.e. $N\ge2$ and $\delta=0,$ is an open problem. More precisely, the distributed null controllability of the system
\begin{equation}\label{zero}
\begin{cases}
\partial_t y-\Delta y=\mathds{1}_\omega(x) v(t,x)& \qquad \mbox{ in } \Omega_T,\\
\partial_t y_\Gamma+\partial_\nu y=0 & \qquad \mbox{ on } \Gamma_T, \\
y_\Gamma=y_{|\Gamma} & \qquad\;\text{on } \Gamma_T, \\
(y,y_\Gamma)|_{t=0}  = (y_0, y_{0,\Gamma}) & \qquad \mbox{ in } \Omega\times\Gamma.
\end{cases}
\end{equation}
Using the Carleman estimate approach, the problematic term
$$s\lambda\int_{\Gamma_T} e^{-2s\alpha}\xi |\partial_\nu \eta| \, |\nabla_\Gamma \varphi_\Gamma|^2\,\d S\,\d t$$
could be absorbed thanks to the surface diffusion term $\delta \Delta_\Gamma \varphi_\Gamma$ when $\delta>0$. Also, in the one-dimensional case $N=1$, this term cancels, and the null controllability of \eqref{zero} holds. The previous Carleman-type techniques do not work when $N\ge2$ and $\delta=0$. Furthermore, the regularity estimates technique fails, since the solution of system \eqref{zero} has only interior regularity of order $H^{\frac{3}{2}}(\Omega)$; in contrast to the case of boundary diffusion ($\delta>0$) where the solution attains $H^4(\Omega)\times H^4(\Gamma)$-regularity. Nevertheless, we expect that distributed null controllability will be maintained. Moreover, similar comments hold for the boundary null controllability problem.

\subsection*{Open problem 4.2 (Boundary Carleman estimate)}
In Lemma \ref{car2}, we have seen a boundary Carleman estimate for the homogeneous system without potentials and source terms. The regularity estimates technique does not allow us to obtain a general Carleman estimate with $L^2$ source terms. Such an estimate would be of much interest, as it implies Lipschitz boundary stability with one observation, in contrast to Theorem \ref{thmstab}.

\subsection*{Open problem 4.3 (General diffusion matrices)}
Consider the system with variable diffusion coefficients
\begin{empheq}[left = \empheqlbrace]{alignat=2}
\begin{aligned}
&\partial_t y - \dv\left(A(t,x) \nabla y\right) = f(t,x), &\text{in } \Omega_T , \\
&\partial_t y_{\Gamma} - \dv_{\Gamma}\left(D(t,x)\nabla_\Gamma y_{\Gamma}\right) +\partial_{\nu}^A y = g(t,x), &\text{on } \Gamma_T, \\
&y_{\Gamma}(t,x) = y_{|\Gamma}(t,x), &\text{on } \Gamma_T, \\
&\left(y,y_{\Gamma}\right)\rvert_{t=0}=\left(y_0, y_{0,\Gamma}\right),   &\Omega\times\Gamma, \label{eq1to4}
\end{aligned}
\end{empheq}
where the diffusion matrices $A$ and $D$ are symmetric and uniformly elliptic with regular coefficients. The conormal derivative with respect to $A$ is given by
$$\partial_{\nu}^A y(t,x) :=\sum\limits_{i,j=1}^N a_{ij}(t,x) \left(\partial_i y\right)_{|\Gamma}\nu_j(x).$$
Most of the results presented above hold for the general system \eqref{eq1to4} provided the assumption
\begin{equation}
    D(t,x): T_x\Gamma \rightarrow T_x\Gamma,
\end{equation}
where $T_x\Gamma$ is the tangent space to $\Gamma$ at $x$. This assumption enables us to simplify some boundary terms using the surface divergence formula
\begin{equation*}
\int_\Gamma \left(\dv_\Gamma X\right)z \,\d S =- \int_\Gamma X \cdot \nabla_{\Gamma} z \,\d S, \qquad z\in H^1\left(\Gamma\right),
\end{equation*}
when $X$ is any $H^1$ (tangential) vector field on $\Gamma$. However, when $D(t,x): T_x\Gamma \rightarrow \mathbb{R}^N$ is a general function, one needs to use the generalized divergence formula (see Proposition 5.4.9 in \cite{HP18})
\begin{equation*}
    \int_{\Gamma} \left(\dv_{\Gamma} X\right) z\, \d S=-\int_{\Gamma} X \cdot \nabla_{\Gamma} z \,\d S+\int_{\Gamma} (\dv_\Gamma\nu)(X \cdot \nu) z \,\d S
\end{equation*}
for $X\in H^1(\Gamma)^N$, which entails some new boundary terms that need to be absorbed.

\subsection*{Open problem 4.4 (General boundary coupling)}
The null controllability of systems with general dynamic boundary conditions and Robin (or Fourier) coupling at the boundary are more challenging. A prototype model is given by
\begin{equation}\label{zero1}
\begin{cases}
\partial_t y-\Delta y=\mathds{1}_\omega(x) v(t,x)& \qquad \mbox{ in } \Omega_T,\\
\partial_t y_\Gamma -\Delta_\Gamma y_\Gamma+\partial_\nu y=0 & \qquad \mbox{ on } \Gamma_T, \\
\sigma y_\Gamma - y_{|\Gamma}-\kappa \partial_{\nu} y=0 & \qquad\;\text{on } \Gamma_T, \\
(y,y_\Gamma)|_{t=0}  = (y_0, y_{0,\Gamma}) & \qquad \mbox{ in } \Omega\times\Gamma,
\end{cases}
\end{equation}
with constants $\sigma \in \mathbb{R}$ and $\kappa \ge 0$. Robin boundary coupling appears, for example, in bulk-surface Allen–Cahn equations; see \cite{CFL19} and \cite{KL21}. Note that for $\kappa=0$ and $\sigma=1$, one recovers the Dirichlet coupling $y_\Gamma = y_{|\Gamma}$ that has been actively studied in the last years in the context of null controllability. In the Robin coupling case, some new terms on the boundary are difficult to absorb with the Carleman weights given by \eqref{w1}.

\end{document}